\documentclass[11pt, reqno]{amsart}
\usepackage{xcolor}
\usepackage{graphicx}
\usepackage{blindtext}
\usepackage{latexsym}
\usepackage[pagewise]{lineno}
\usepackage[utf8]{inputenc}
\usepackage{exscale}
\usepackage{amsmath}
\usepackage{amssymb}
\usepackage{amsfonts}
\usepackage{mathrsfs}
\usepackage{amsbsy}
\usepackage{relsize}
\usepackage{bbold}
\usepackage[final]{pdfpages}
\usepackage[colorlinks, linkcolor=purple, citecolor=blue, urlcolor=blue, hypertexnames=false]{hyperref}
\usepackage{comment}
\PassOptionsToPackage{reqno}{amsmath}
\usepackage{esint} 
\usepackage{amssymb} 
\usepackage{stmaryrd}
\usepackage{cite}
\usepackage{geometry}
\newcommand{\R}{\mathbb{R}}

\newcommand{\beq}{\begin{eqnarray}}
\newcommand{\eeq}{\end{eqnarray}}
\newcommand{\bq}{\begin{equation}}
\newcommand{\eq}{\end{equation}}
\newcommand{\beqn}{\begin{eqnarray*}}
\newcommand{\eeqn}{\end{eqnarray*}}
\newcommand{\bex}{\begin{exo}}
\newcommand{\eex}{\end{exo}}
\newcommand{\ben}{\begin{enumerate}}
\newcommand{\een}{\end{enumerate}}

\newcommand{\bbN}{\mathbb{N}}
\newcommand{\bbR}{\mathbb{R}}

\newtheorem{th1}{{\bf Theorem}}[section]
\newtheorem{thm}[th1]{{\bf Theorem}}
\newtheorem{lem}[th1]{{\bf Lemma}}
\newtheorem{prop}[th1]{{\bf Proposition}}

\newtheorem{rem}[th1]{\bf Remark}

\newtheorem{defi}[th1]{\bf Definition}

\usepackage{lipsum} 
\usepackage{tikz,xcolor}

\definecolor{lime}{HTML}{A6CE39}
\DeclareRobustCommand{\orcidicon}{
	\begin{tikzpicture}
	\draw[lime, fill=lime] (0,0) 
	circle [radius=0.16] 
	node[white] {{\fontfamily{qag}\selectfont \tiny ID}};
	\draw[white, fill=white] (-0.0625,0.095) 
	circle [radius=0.007];
	\end{tikzpicture}
	\hspace{-2mm}
}
\foreach \x in {A, B, C}{\expandafter\xdef\csname orcid\x\endcsname{\noexpand\href{https://orcid.org/\csname orcidauthor\x\endcsname}
			{\noexpand\orcidicon}}
}

\usepackage{bm}

\author[T. A. Enaoufal \& T. Saanouni]{Taif Abdullah Enaoufal \& Tarek Saanouni$^*$\orcidC}
\thanks{* Corresponding author.}

\address[T. Saanouni]{Department of Mathematics, College of Science, Qassim University, Buraydah, Kingdom of Saudi Arabia.}
\email{\sl\textcolor{blue}{t.saanouni@qu.edu.sa}}

\address[T. A. Enaoufal]{Department of Mathematics, College of Science, Qassim University, Buraydah, Kingdom of Saudi Arabia.}

\email{\sl\color{blue}{46121544@qu.edu.sa}}



\title[BINLS]{Biharmonic non-linear Schr\"odinger equation with an unbounded inhomogeneous term}

\date{\today}
\begin{document}
\begin{abstract}
This paper is devoted to the analysis of a focusing nonlinear biharmonic Schr\"odinger equation in the presence of an unbounded growing up inhomogeneous term. The first main contribution of this work is the derivation of an inhomogeneous Gagliardo-Nirenberg inequality adapted to the unbounded weight, which provides the necessary control over the nonlinear term in terms of Sobolev norms. Building on this inequality, we then investigate the long-time behavior of solutions and establish a sharp dichotomy: solutions with initial data below the ground state energy either exist globally in time or experience finite-time blow-up. A distinctive feature of our results is that the analysis of the unbounded inhomogeneous term requires the imposition of radial symmetry on the initial data, which allows us to exploit certain Strauss type Sobolev estimates that would not hold in the general non-radial case. This work complements previous studies on biharmonic Schr\"odinger equations with singular inhomogeneities, highlighting both the challenges and the new phenomena that arise when the nonlinearity is weighted by a growing up unbounded function, which broke the space translation invariance of the standard homogeneous associated equation. 
\end{abstract}
\maketitle
\vspace{ 1\baselineskip}
\renewcommand{\theequation}{\thesection.\arabic{equation}}
\section{Introduction}
It is the purpose of this note, to investigate the focusing biharmonic inhomogeneous Schr\"odinger equation

\begin{align}
\left\{
\begin{array}{ll}
\textnormal{i}\partial_t v-\Delta^2  v=- |x|^{b}|v|^{q-1}v ;\\
v(0,\cdot)=v_0.
\label{S}
\end{array}
\right.
\end{align}

Here and throughout this work, we denote by the bi-Laplacian operator the expression
\[
\Delta^2 := \Delta(\Delta\,),
\]
which corresponds to applying the Laplacian operator twice. The function \(v\) represents the wave function, assumed to be complex-valued and depending on the variables \((t, x) \in \mathbb{R} \times \mathbb{R}^N\), where \(N\) is an integer satisfying \(N \geq 1\). The sign of the source gives a \emph{focusing} regime. Moreover, the notation \(|\cdot|^{b}\) designates an unbounded inhomogeneous term, for some positive real number \(b > 0\).

The theoretical foundation of this framework can be traced back to the seminal studies of \cite{Karpman} and later \cite{Karpman 1}. Their investigations were primarily driven by the need to extend the classical nonlinear Schr\"odinger equation in order to provide a more accurate description of high-intensity laser pulse propagation in bulk media characterized by Kerr-type nonlinearities. They observed that, in such regimes, the conventional paraxial approximation becomes inadequate, as it fails to account for important higher-order dispersive contributions.
To overcome this limitation, they systematically incorporated a fourth-order dispersion term into the governing evolution equation. This modification resulted in a more complete theoretical model, capable of capturing new classes of nonlinear wave behavior—most notably, mechanisms leading to the stabilization of solitons that would otherwise undergo collapse in the standard formulation. The inclusion of this higher-order term therefore yielded valuable insights into beam propagation dynamics that could not be obtained from the traditional second-order theory.
Within the setting of laser–plasma interactions, the source term appearing in this generalized model is interpreted as a nonlinear potential that both arises from and dynamically affects localized variations in the electron density. This physical interpretation of the source term is consistent with, and further supported by, the earlier theoretical developments presented in \cite{brg}.

The Schr\"odinger problem \eqref{S} admits two primary conservation laws: the preservation of \emph{mass} and that of \emph{energy}, which remain invariant under the flow of the equation.

\begin{align}
 \int_{\R^N}|v(t,x)|^2\,dx &:=\texttt{M}(v(t)) = \texttt{M}(v_0);\label{mass}\tag{Mass}\\   
 \int_{\R^N}\Big(|\Delta v(t,x)|^2-\frac{2}{1+q}|v(t,x)|^{1+q}|x|^{b}\Big)dx&:=\texttt{E}(v(t)) = \texttt{E}(v_0).\label{nrg}\tag{Energy}
\end{align}

The inhomogeneous Sobolev space \( H^2_{rd} \) is commonly called the \emph{energy space} because it offers the lowest level of regularity at which the conservation of mass and energy can be rigorously justified. Now, if \( v \) is a solution to equation \eqref{S}, then the next rescaled family is also a solution.   
\begin{align}
v_\omega(t,x) = \omega^{\frac{4+b}{q-1}}\, v(\omega^4 t, \omega x), \quad \omega > 0, 
\label{scl}
\end{align}
Moreover, the identity 
\[
\|v_\omega(t)\|_{\dot{H}^{s_c}} = \|v(\omega^4 t)\|_{\dot{H}^{s_c}}
\]
defines the so-called \emph{critical index} \( s_c:=\frac N2-\frac{4+b}{q-1} \), which characterizes the unique homogeneous Sobolev space invariant under the above scaling. The $\dot{H}^s$-critical exponents associated with the Schr\"odinger equation \eqref{S} is given by  

\begin{align}
s_c = s \;\Leftrightarrow\;
q^{e}_s :=
\begin{cases}
1 + \dfrac{8 + 2b}{N - 2s}, & \text{if }\quad N > 2s, \\[1.2ex]
\infty, & \text{if }\quad 1 \leq N \leq 2s.
\end{cases}
\label{crit-exp}
\end{align}
Moreover, $q_m := q_0^e$ denotes the mass-critical exponent, while \( q^{e}_s \) represents the $\dot{H}^s$-critical exponent, corresponding respectively to the cases where the scaling leaves the $L^2$ and $\dot{H}^s$ norms invariant.

In recent years, considerable research interest has focused on the following inhomogeneous fourth-order nonlinear Schr\"odinger equation:

\begin{align}
    \textnormal{i}\partial_t v + \Delta^2 v = \pm |x|^{-b}|v|^{q-1}v, \quad b > 0. \label{S'}
\end{align}

The initial mathematical foundation for this problem was laid by Guzmán and Pastor~\cite{gp}. Their work proved the local well-posedness of the system in the Sobolev space $H^2(\mathbb{R}^N)$ for dimensions $N \geq 3$, with the parameter $b$ in the interval $0 < b < \min\left\{\frac{N}{2}, 4\right\}$. Their analysis was valid for nonlinearity exponents $q$ satisfying $\max\left\{0, \frac{2(1-b)}{N}\right\} < q-1$, alongside the scaling constraint $(N-4)(q-1) < 8 - 2b$. Furthermore, they established global well-posedness for the mass-subcritical and mass-critical cases, characterized by $\min\left\{\frac{2(1-b)}{N}, 0\right\} < q-1 \leq \frac{2(4-b)}{N}$, subject to certain auxiliary conditions.
This theory was subsequently refined by Cardoso, Guzmán, and Pastor~\cite{cgp}, who extended the local well-posedness framework to the homogeneous Sobolev spaces $\dot{H}^s \cap \dot{H}^2$ for $N \geq 5$, $0 < s < 2$, and $0 < b < \min\left\{\frac{N}{2}, 4\right\}$. Their result required the exponent to satisfy $\max\left\{\frac{8 - 2b}{N}, 1\right\} < q-1 < \frac{8 - 2b}{N-4}$. Limitations pertaining to lower spatial dimensions and a lower bound on $q$ were later circumvented by the authors of~\cite{lz}, who leveraged advanced bilinear Strichartz estimates within the context of Besov spaces.
The global dynamics for initial data of small energy were investigated in the inter-critical regime in~\cite{gp} and for the energy-critical case in~\cite{gp2}. In the mass-critical setting, a substantial body of literature has been devoted to the analysis of finite-time blow-up phenomena for solutions with negative energy, a line of inquiry initiated by the seminal works~\cite{cow,vdd2,rbts}.
Significant advances have also been made in understanding the long-term asymptotic behavior (scattering) of global solutions in the repulsive inter-critical regime. The first author established in~\cite{st4} that for spherically symmetric initial data, such solutions not only exist globally but also scatter. This radial symmetry assumption was later removed in~\cite{cg2}, generalizing the scattering result to arbitrary (non-radial) initial data. A complete characterization of the dynamics below the ground state energy level, presenting a sharp dichotomy between global existence and finite-time blow-up, is provided in~\cite{sg2}. For the defocusing variant of the problem, energy scattering was demonstrated by the first author in~\cite{st1}.
A local existence theory for the energy-critical case was developed by the first author in~\cite{sg1,sp1}. A comprehensive series of studies by An et al.~\cite{akr1,akr2,akr3,akr4} undertook a systematic investigation of the inhomogeneous biharmonic NLS \eqref{S}. These works rigorously established local and global well-posedness, together with the standard property of continuous solution dependence on the initial data in $H^s(\bbR^N)$. Their results hold for all spatial dimensions $N \in \bbN$ and regularity indices $s$ satisfying $0 \leq s < \min\left\{\frac{2+N}{2}, \frac{3N}{2}\right\}$, under the parameter restrictions $0 < b < \min\left\{4, N, \frac{3N}{2}-s, \frac{N}{2}+2-s\right\}$ and $0 < q < q_s^e$.
In the critical case where $q = q_s^e$, these works also demonstrated that the corresponding Cauchy problem is locally well-posed in $H^s(\bbR^N)$, provided specific conditions are met. This conclusion applies when $N \geq 3$, the regularity index obeys $1 \leq s < \frac{N}{2}$, and the parameter $b$ is constrained to $0 < b < \min\left\{4, 2 + \frac{N}{2} - s \right\}$. This holds if $q$ is an even integer, or more generally, if the condition $q-1 > \lceil s \rceil - 2$ is satisfied\footnote{Here, $\lceil \cdot \rceil$ denotes the integer part function.}.
This article aims to systematically investigate the long-term dynamics of energy solutions with initial data below the ground state threshold, establishing a characterization of the dichotomy between global existence and finite-time blow-up in the spirit of the foundational work by \cite{hr}. Our analysis reveals that in the energy-subcritical regime, global existence follows from careful variational arguments that identify the associated ground state as a minimizer of an inhomogeneous Gagliardo-Nirenberg type inequality. For the inter-critical case, we demonstrate finite-time blow-up through a refined localized variance identity methodology, extending the techniques pioneered in \cite{bl}. The temporal structure of blow-up exhibits regime-dependent behavior: while mass-critical solutions may develop singularities in either finite or infinite time, the mass-supercritical regime universally exhibits finite-time breakdown. The primary mathematical challenge arises from the strongly singular inhomogeneous coefficient $|\cdot|^{b}$ for $b>0$, which renders conventional methods for $b<0$ inapplicable. Indeed, we can't decompose the integrals on the unit ball and its complementary in order to use the integrity of such a term in the case $b<0$, but also we can't use Hardy type estimates nor Lorentz spaces with $|x|^{b}\in L^{-\frac Nb,\infty}$; we overcome this obstruction by developing novel Strauss-type weighted estimates under spherical symmetry assumptions. To the best of our knowledge, this work presents the first comprehensive treatment of this singular problem, filling a significant gap in the current literature. Note that in the case of a Laplace operator in \eqref{S}, namely the INLS, a series of works exist in the literature \cite{CG-DCDS-B, Chen, Chen-CMJ, CG-AM, dms}.
For simplicity, we give some notations. Let the standard Lebesgue and Sobolev spaces and norms

\begin{align}
    L^r&:=L^r({\R^N}),\quad H^2:=H^2({\R^N}),\\
    \|\cdot\|_r&:=\|\cdot\|_{L^r},\quad\|\cdot\|:=\|\cdot\|_2,\quad \|\cdot\|_{H^2} := \Big(\|\cdot\|^2 + \|\Delta\cdot\|^2\Big)^\frac12.
\end{align}
Take also the radial Sobolev space $H^2_{rd}:=\{f\in H^2,\quad f(x)=f(|x|)\}$ and $T^{\max}>0$ be the lifespan of an eventual solution to \eqref{S}. Finally, we define the real numbers

\begin{align}
D:=\frac{Nq-N-2b}4,\quad E:=1+q-D. \label{ed}
\end{align}

The next sub-section contains the contribution of this manuscript.
\subsection{Main results}

We start with an inhomogeneous fourth-order Gagliardo-Nirenberg type estimate.

\begin{thm}\label{gag}
Let $N\geq2$, $b>0$ and $1+\frac{2b}{N-1}< q< q^e$. Thus, 
\begin{enumerate}
\item
there exists a sharp constant $\texttt{C}_{opt}>0$, such that for all $v\in H^2_{rd}$,

\begin{align}
\int_{\R^N}|v|^{1+q}|x|^{b}\,dx\leq \texttt{C}_{opt}\|v\|^E\|\Delta v\|^D;    \label{gag2}
\end{align}
\item
there exists $\zeta\in H^2_{rd}$ a minimizing of the problem

$$\frac1{\texttt{C}_{opt}}=\inf_{0\neq v\in H^2_{rd}}\frac{\|v\|^E\|\Delta v\|^D}{\int_{\R^N}|v|^{1+q}|x|^{b}\,dx},$$
 such that 
 \begin{align}\label{grnd}
\zeta+\Delta^2\zeta-|x|^{b}|\zeta|^{q-1}\zeta=0,\quad0\neq \zeta\in H^2_{rd};
\end{align}
\item
moreover,
\begin{align}\label{part32}
\texttt{C}_{opt}=\frac{1+q}E\big(\frac ED\big)^{\frac{D}2}\|\zeta\|^{-(q-1)};
\end{align}
\item
furthermore, the next Pohozaev identities holds

\begin{align}\label{poh}
\int_{\R^N}|\zeta|^{1+q}|x|^b\,dx=\frac{1+q}{E}\,\texttt{M}(\zeta)=\frac{1+q}{D}\|\Delta\zeta\|^2.
\end{align}
\end{enumerate}
\end{thm}

In view of the results stated in the above theorem, some comments are in order.

\begin{itemize}
\item[$\spadesuit$]
The proof makes use of several ideas originally introduced by M. I. Weinstein \cite{wns}.
\item[$\spadesuit$]
The proof uses a compact Sobolev embedding given in Lemma \ref{compact}, which is prove in the appendix.
\item[$\spadesuit$]
The above result is essential to develop a variational analysis of energy solutions to \eqref{S}, in order to investigate the global versus non-global existence of solutions. This will be the objective of the next Theorem.
\item[$\spadesuit$]
In the complementary case $q<1+\frac{2b}{N-1}$, we have $\texttt{C}_{opt}=\infty$. Indeed, taking $\xi_n:=\xi(\cdot-n)$, where $\xi$ is a radial smooth function with support subset of $[0,1]$. Hence, 

\begin{align}
    \frac{\int_{\R^N}|\xi_n|^{1+q}|x|^{b}\,dx}{\|\xi_n\|^E\|\Delta\xi_n\|^D}  
    &\gtrsim    \frac{n^{-1+N+b}}{n^\frac{(D+E)(N-1)}2} \gtrsim    n^{1+\frac{2b}{N-1}-q}  \to\infty,\quad\mbox{as}\quad n\to\infty.
\end{align}
\end{itemize}

Using Theorem \ref{gag}, we obtain a dichotomy of global versus non-global existence of energy solutions to \eqref{S}.

\begin{thm}\label{glb}
Let $N\geq2$, $b\geq0$ and $\max\{q_m,1+\frac{2b}{N-1}\}\leq q\leq q^e$. Take a local solution to \eqref{S} denoted by ${v}\in C\big([0,T],H^2_{rd}\big)$, satisfying

\begin{align} 
{\texttt{E}({v_0})}^{s_c}{\texttt{M}({v_0})}^{2-s_c}<\texttt{E}(\zeta)^{s_c}\texttt{M}(\zeta)^{2-s_c}.\label{ss}
\end{align}

\begin{enumerate}
\item
 Assume that
 
\begin{align} \label{ss1}
\|\Delta {v_0}\|^{s_c}\|{v_0}\|^{2-s_c}<\|\Delta\zeta\|^{s_c}\|\zeta\|^{2-s_c},
\end{align}

then the solution is global whenever $q<q^e$ and is uniformly bounded if $q=q^e$.
\item
 Assume that $q\leq9$ and 
\begin{align} \label{ss2}
\|\Delta {v_0}\|^{s_c}\|{v_0}\|^{2-s_c}>\|\Delta\zeta\|^{s_c}\|\zeta\|^{2-s_c},
\end{align}
then, the solution blows-up in finite time whenever $q_m<q$ and blows-up in infinite time if $q=q_m$.
 
\end{enumerate}
\end{thm}

Given the results presented in the preceding theorem, a few remarks are warranted.

\begin{itemize}
\item[$\spadesuit$]
The assumption $q_m\leq q\leq q^e$ is equivalent to $0\leq s_c\leq2$. 
\item[$\spadesuit$]
The restriction $q\geq 1+\frac{2b}{N-1}$ is due to the method used here, precisely of some Strauss type estimates, namely \eqref{frcs}-\eqref{frcs'}.
\item[$\spadesuit$]
The condition $q\leq9$ is needed in the proof of the finite time blow-up for some use of Young estimate, precisely in \eqref{5.7}. Indeed, an estimate of the potential term with Strauss type estimates and the mass conservation law, gives a term $R^{-\delta}\|\nabla v \|^{\frac{q-1}{2}}\lesssim R^{-\delta}\|\Delta v\|^\frac{q-1}{4}$. This requires $\frac{q-1}{4}\leq 2$, in order to estimate to write by Young estimate $R^{-\delta}\|\Delta v\|^\frac{q-1}{4}\lesssim R^{-\delta_1}+R^{-\delta_2}\|\Delta v\|^2$. These terms are quiet absorbable.
\item[$\spadesuit$]
The assumption $q\leq9$, may be removed whenever $q^e\leq9$, which reads $b\leq4(N-5)$. 
\item[$\spadesuit$]
In the energy critical case, we have an uniform bound $\displaystyle\sup_{t<T^{\max}}\|v(t)\|_{H^2}<\infty$. This is namely given in \eqref{4.2} via the mass conservation law.
\item[$\spadesuit$]
In the mass-critical regime, the finite time blow-up of energy solutions to \eqref{S}, is still open even for the standard case $b=0$, with radial data, see \cite[Theorem 3]{bl}.
\item[$\spadesuit$]
In the proof, we establish that the assumptions \eqref{ss}-\eqref{ss1}, but also \eqref{ss}-\eqref{ss2} are invariant under the flow of \eqref{S}.
\item[$\spadesuit$]
The scattering of global solutions, is treated in a paper in progress. 
\end{itemize}

The next sub-section contains some standard tools.
\subsection{Useful estimates}

 We start with the next Hardy estimate \cite[Theorem 0.1]{majb}.

\begin{lem}
    Let $1<r<\infty$, $0<s<\frac Nr$ and $u\in \dot W^{s,r}(\R^N)$, then
    \begin{align}
        \||\cdot|^{-s}u\|_r\leq C_{N,s,r}\||\nabla|^su\|_{r}.\label{hrd}
    \end{align}
\end{lem}

Sobolev injections \cite{co} will be useful.
\begin{lem}\label{sblv}
For $N\geq2$, hold
\begin{enumerate}
\item
$H^2\hookrightarrow L^q$ for every $q\in[2,\frac{2N}{N-4}]$ if $N\geq5$ and every $2\leq q<\infty$ if $N\leq4$;
\item
 $H^2_{rd} \hookrightarrow\hookrightarrow L^q$ is compact for every $q\in(2,\frac{2N}{N-4})$ if $N\geq5$ and every $2<q<\infty$ if $N\leq4$;
\item
for any $u\in H^1_{rd}$ and any $\frac12\leq s<1$, holds

\begin{align}\label{frcs}
\sup_{x\in\R^N}|x|^{\frac{N-2s}2}|u(x)|\leq C_{N,s}\|u\|^{1-s}\|\nabla v\|^s.\end{align}
\item
for any $\frac12< s<\frac N2$ and any $u\in\dot H^s_{rd}$, holds

\begin{align}\label{frcs'}
\sup_{x\in\R^N}|x|^{\frac{N-2s}2}|u(x)|\leq C_{N,s}\||\nabla|^s u\|.\end{align}
\end{enumerate}
\end{lem}

Finally, we give a helpful compact Sobolev embedding, proved in the appendix.

\begin{lem}\label{compact}
Let $N\geq1$ and $b\geq0$ and $2<1+q-\frac{2b}{N-1}<2+\frac{8}{N-4}$, where $\frac{1}{N-4}=\infty$, for $1\leq N\leq4$. Thus, the next injection is compact
$$H^2_{rd}\hookrightarrow\hookrightarrow L^{1+q}(|x|^{b}\,dx).$$
\end{lem}

The rest of this paper is organized as follows. Section \ref{sec3} proves the Gagliardo-Nirenberg type estimate in Theorem \ref{gag2}. Section \ref{sec4} establishes the global/non-global existence of energy solutions stated in Theorem \ref{glb}.

\section{Gagliardo-Nirenberg type estimate}\label{sec3}

This section is devoted to establish Theorem \ref{gag2}. This proof incorporates certain concepts introduced by M. I. Weinstein \cite{wns}. Let us denote by $(\zeta_n)\in H^2_{rd}$ be a minimizing sequence of \eqref{gag}, namely 

\begin{align}
\frac1{\texttt{C}_{opt}}
&:=\lim_n\frac{\|{\zeta_n}\|^E\|\Delta {\zeta_n}\|^D}{\int_{\R^N}|{\zeta_n}|^{1+q}\,|x|^{b}\,dx}\nonumber\\
&:=\lim_n\texttt{K}({\zeta_n}).\label{3.1}
\end{align}

Taking the scaling $\zeta^{{\kappa},{\nu}}:={\kappa}\zeta({\nu} \cdot)$, for ${\kappa},{\nu}\in\R$, we compute

\begin{align}
\|\Delta \zeta^{{\kappa},{\nu}}\|&={\kappa}{\nu}^{2-\frac N2}\|\Delta \zeta\|;\label{3.2}\\
\|\zeta^{{\kappa},{\nu}}\|&={\kappa}{\nu}^{-\frac N2}\|\zeta\|;\label{3.3}\\
\int_{\R^N}|\zeta^{{\kappa},{\nu}}|^{1+q}\,|x|^{b}\,dx&={\kappa}^{1+q}{\nu}^{-N-b}\int_{\R^N}|\zeta|^{1+q}\,|x|^{b}\,dx.\label{3.4}
\end{align}

Hence,

\begin{align}
\texttt{K}(\zeta^{{\kappa},{\nu}})=\texttt{K}(\zeta).\label{3.5}
\end{align}

Letting

\begin{align}
    {\nu}_n&:=\Big(\frac{\|{\zeta_n}\|}{\|\Delta  {\zeta_n}\|}\Big)^\frac12,\quad {\kappa}_n:=\frac{\|{\zeta_n}\|^{\frac N{4}-1}}{{\|\Delta  {\zeta_n}\|}^\frac N{4}},\quad   \varphi_n:={\zeta_n}^{{\kappa}_n,{\nu}_n},
\end{align}

it follows that by \eqref{3.2}, \eqref{3.3} and \eqref{3.5},

\begin{align}
{\frac1{\texttt{C}_{opt}}}=\lim_n\texttt{K}(\varphi_n) ,\quad 1=\|\varphi_n\|=\|\Delta \varphi_n\| .    
\end{align}

Take ${\varphi}\in H^2$ the weak limit of $\varphi_n$ in $H^2$. Lemma \ref{compact} ensures that for a subsequence, we have

\begin{align}
\frac1{\texttt{K}(\varphi_n)}={\int_{\R^N}|\varphi_n|^{1+q}\,|x|^{b}\,dx}\mathop{\longrightarrow}\limits_{n \to \infty}{\int_{\R^N}|{\varphi}|^{1+q}\,|x|^{b}\,dx}.    \label{3.6}
\end{align}

Moreover, by the lower semi continuity of ${H^2}$ norm, it follows that $\max\{ \|\Delta{\varphi}\|,\|{\varphi}\|\}\leq1.$ Thus, if we assume that $\|{\varphi}\|\|\Delta {\varphi}\|<1$, taking account of \eqref{3.6}, we get the contradiction

\begin{align}
    \texttt{K}({\varphi})<\frac1{\int_{\R^N}|{\varphi}|^{1+q}\,|x|^{b}\,dx}= {\frac1{\texttt{C}_{opt}}}.
\end{align}

Hence, $\|{\varphi}\|=\|\Delta{\varphi}\|=1$ and so, $\varphi_n\rightarrow{\varphi}$ {in} $H^2$ and

\begin{align}
{\frac1{\texttt{C}_{opt}}}&=\texttt{K}({\varphi})=\frac1{\int_{\R^N}|x|^{b}|{\varphi}|^{1+q}\,dx}.
\end{align}

Moreover, since the minimizer satisfies the Euler equation

\begin{align}
\partial_\lambda \texttt{K}\big({\varphi}+\lambda\psi\big)_{|\lambda=0}=0,\quad\forall \psi\in C_0^\infty\cap H^2,    
\end{align}

hence ${\varphi}$ satisfies the elliptic equation

\begin{align}\label{euler2}
D\Delta^2\varphi+E\varphi-\frac{1+q}{\texttt{C}_{opt}}|x|^{b}|\varphi|^{q-1}\varphi=0.
\end{align}

A direct calculus implies that, the the scaled function

\begin{align}
{\varphi}=\zeta^{{\kappa},{\nu}},\quad {\nu}=\Big(\frac ED\Big)^\frac1{4}\quad\mbox{and}\quad {\kappa}=\Big(\big(\frac{E}D\big)^\frac{b}{4}\frac{ E{\texttt{C}_{opt}}}{(1+q)}\Big)^\frac1{q-1},    
\end{align}

satisfies \eqref{grnd}. Hence, by the equalities $\|{\varphi}\|=1={\kappa}{\nu}^{-\frac N2}\|\zeta\|,$ we get \eqref{part32}. Finally, we prove 
the Pohozaev identities in \eqref{poh}. We test \eqref{gag} with $\bar\zeta$ and we integrate, 

\begin{align}
\|\Delta\zeta\|^2+\|{\zeta}\|^2={\int_{\R^N}|\zeta|^{1+q}|x|^b\,dx}.  \label{3.6'}  
\end{align}

We call action, the functional

\begin{align}
\texttt{S}({\zeta}):=\|\Delta{\zeta}\|^2+\|{\zeta}\|^2-\frac2{1+q}{\int_{\R^N}|\zeta|^{1+q}|x|^b\,dx}.    
\end{align}

Since \eqref{gag} gives $\texttt{S}'({\zeta})=0$, we have for any $\alpha,\beta\in\R$,

\begin{align}
0&=\partial_\lambda\Big(\texttt{S}({\zeta}^{\lambda^\alpha,\lambda^\beta})\Big)_{|\lambda=1}:=\partial_\lambda\Big(S\big(\lambda^\alpha{\zeta}(\lambda^\beta\cdot)\big)\Big)_{|\lambda=1}.
\end{align}
 
 Using \eqref{3.2}, \eqref{3.3} and \eqref{3.4}, we write
 
\begin{align}
\|\Delta{\zeta}^{\lambda^\alpha,\lambda^\beta}\|&={\lambda^{\alpha+\beta(2-\frac N2)}}\|\Delta {\zeta}\|;\\
\|{\zeta}^{\lambda^\alpha,\lambda^\beta}\|&={\lambda^{\alpha-\frac{N\beta}2}}\|{\zeta}\|,\\
{\int_{\R^N}|{\zeta}^{\lambda^\alpha,\lambda^\beta}|^{1+q}|x|^b\,dx}&=\lambda^{\alpha(1+q)-\beta(N+b)} {\int_{\R^N}|\zeta|^{1+q}|x|^b\,dx}.
\end{align}

Hence,

\begin{align}
\partial_\lambda\Big(\texttt{S}({\zeta}_{\alpha,\beta}^\lambda)\Big)_{|\lambda=1}
&=2(\alpha+\beta(2-\frac N2))\|\Delta{\zeta}\|^2+2(\alpha-\beta\frac N2))\|u\|^2\nonumber\\
&-2\frac{\alpha(1+q)-\beta(N+b)}{1+q}{\int_{\R^N}|\zeta|^{1+q}|x|^b\,dx}.
\end{align}

Letting $\alpha=N$ and $\beta=2$, it follows that

\begin{align}
    \|\Delta{\zeta}\|^2=\frac{D}{1+q}{\int_{\R^N}|\zeta|^{1+q}|x|^b\,dx}.\label{3.7}
\end{align}

Eventually, by \eqref{3.6'} and \eqref{3.7}, we get \eqref{poh} and the proof is closed.
\section{Global/non-global existence of energy solutions}\label{sec4}

This section proves Theorem \ref{glb} about a dichotomy of global versus non-global existence of energy solutions under the ground state threshold. We start with the next intermediary result.

\begin{lem}\label{stbl}
Let $v\in C\big([0,T],H^2_{rd}\big)$ be a local solution to \eqref{S}. Assume that \eqref{ss} holds, then
\begin{enumerate}
\item
if \eqref{ss1} holds, then, for any $0\leq t\leq T$,
\begin{align} \label{ss1'}
\|\Delta {v(t)}\|^{s_c}\|{v_0}\|^{2-s_c}<\|\Delta\zeta\|^{s_c}\|\zeta\|^{2-s_c};
\end{align}
\item
if \eqref{ss2} holds, then, for any $0\leq t\leq T$,
\begin{align} \label{ss2'}
\|\Delta {v(t)}\|^{s_c}\|{v_0}\|^{2-s_c}>\|\Delta\zeta\|^{s_c}\|\zeta\|^{2-s_c};
\end{align}
\end{enumerate}
\end{lem}
\begin{proof}

Using \eqref{gag2} via the energy conservation, one has

\begin{align}\label{xxx}
 \texttt{E}(v_0)
&= \|\Delta v\|^2 - \frac2{1+q}\int_{\R^N}|v|^{1+q}|x|^{b}\,dx\nonumber\\ 
&\geq \|\Delta v\|^2 - \frac{2\texttt{C}_{opt}}{1+q}\|v_0\|^E\|\Delta v\|^D.
\end{align}

We denote for short a real number, as follows

\begin{align}
{\varkappa}:=\frac{2\texttt{C}_{opt}}{1+q}{\|v_0\|}^E.
\end{align}

Thus, for any $0\leq t\leq T$,

\begin{align}\label{eqq}
\digamma\big(\|\Delta v(t)\|^2\big):=\|\Delta v(t)\|^2 - {\varkappa}\big(\|\Delta v(t)\|^2\big)^\frac D2\leq \texttt{E}(v_0).
\end{align}
The above real function has a maximum 

\begin{align}
\digamma\big({\tau}\big):=\digamma\Big(\big( \frac{2}{{\varkappa} D}\big)^{\frac{2}{D-2}}\Big)=\tau\Big( 1-\frac{2}{D}\Big).\label{tau}
\end{align}

Taking account of \eqref{poh}, it follows that

\begin{align}
\texttt{E}(\zeta)=\frac{D-2}{D}\|\Delta\zeta\|^2= \frac{D-2}{E}\texttt{M}(\zeta) .    
\end{align}

So, taking account of \eqref{ss}, we write

\begin{align}\label{x}
\texttt{E}(v_0)< \frac{D-2}{E}\texttt{M}(\zeta)^{\frac2{s_c}}{\texttt{M}(v_0)}^{\frac{s_c-2}{s_c}}.
\end{align}

Now, by \eqref{part32} via \eqref{tau} and the identity $s_c(q-1)=2(D-2)$, we get

\begin{align}
\digamma({\tau})
&=\Big(\frac{1+q}{{\texttt{C}_{opt}}{\texttt{M}(v_0)}^{\frac{E}2}D}\Big)^\frac2{D-2}\Big(1-\frac2D\Big)\nonumber\\
&=\Big((\frac{E}{D})^{1-\frac D2}(\texttt{M}(\zeta))^\frac{q-1}2({\texttt{M}(v_0)})^{-\frac{E}2}\Big)^\frac2{D-2}\Big(1-\frac2D\Big)\nonumber\\
&=\frac{D-2}E\Big((\texttt{M}(\zeta))^\frac{q-1}2({\texttt{M}(v_0)})^{-\frac{E}2}\Big)^\frac2{D-2}\nonumber\\
&= \frac{D-2}{E}{\texttt{M}(v_0)}^{\frac{s_c-2}{s_c}}\texttt{M}(\zeta)^{\frac2{s_c}}.\label{xx}
\end{align}

Collecting \eqref{x}, \eqref{xx} and \eqref{eqq}, yields

\begin{align}\label{ss3} 
\digamma(\|\Delta v(t)\|^2) \leq {\texttt{E}(v_0)}<\digamma({\tau}). \end{align}

Moreover, by \eqref{tau} via \eqref{xx}, we have

\begin{align}
    {\tau}=\frac DE\texttt{M}(\zeta)^{\frac2{s_c}}\texttt{M}(v_0)^{\frac{s_c-2}{s_c}}.\label{to}
\end{align}

Now, the assumption \eqref{ss1} via \eqref{poh} and \eqref{to}, reads

\begin{align}
\|\Delta v_0\|^2
<\Big(\frac{\texttt{M}(\zeta)}{{\texttt{M}(v_0)}}\Big)^\frac{2-s_c}{s_c}\|\Delta\zeta\|^2=\frac DE\Big(\frac{\texttt{M}(\zeta)}{{\texttt{M}(v_0)}}\Big)^\frac{2-s_c}{s_c}\texttt{M}(\zeta)={\tau}.\label{4.1}
\end{align}

Taking account of \eqref{ss3} via \eqref{4.1} and a continuity argument, yields

\begin{align}
\|\Delta v(t)\|^2<{\tau}=\Big(\frac{\texttt{M}(\zeta)}{{\texttt{M}(v_0)}}\Big)^\frac{2-s_c}{s_c}\|\Delta\zeta\|^2    ,\quad\mbox{for every}\quad t\in [0, T].
\end{align}

 The inequality \eqref{ss1'} is proved and \eqref{ss2'} follows with a similar way.

\end{proof}

\subsection{Global existence}
Thanks to \eqref{ss1'}, we have 

\begin{align} 
\|\Delta {v}\|_{L^\infty(L^2)}\leq \Big(\frac{\texttt{M}(\zeta)}{{\texttt{M}(v_0)}}\Big)^\frac{2-s_c}{2s_c}\|\Delta\zeta\|.\label{4.2}
\end{align}

The global existence follows by the bound \eqref{4.2}.

\subsection{Non-global existence}
Here, we prove the second part of Theorem \ref{glb}. Let $\chi:\R^N\to\R$ be a convex smooth function and the real function defined on $[0,T^{\max})$, by

\begin{align}
    M:=M_\chi(v):t\mapsto2\int_{\R^N}\nabla\chi(x)\cdot\Im\big(\nabla v(t,x)\bar v(t,x)\big)\,dx.\label{mrwtz}
\end{align}

We adopts the convention that repeated indexes are summed. Let us start with a Morawetz type identity \cite[Lemma 3.3]{st4}.

\begin{lem}\label{mrw1}

Then, the following identity hold on $[0,T^{\max})$,

\begin{align}
\frac{\partial}{\partial t}M
&=-2\int_{\R^N}\Big(2\partial_{jk}\Delta\chi\partial_jv\partial_k\bar v-\frac12(\Delta^3\chi)|v|^2-4\partial_{jk}\chi\partial_{ik}u\partial_{ij}\bar v\nonumber\\
&+\Delta^2\chi|\nabla v|^2+\frac{q-1}{1+q}(\Delta\chi)|v|^{1+q}|x|^{b}-\frac2{1+q}|v|^{1+q}\nabla\chi\cdot\nabla(|x|^{b})\Big)\,dx.\label{mrw}
\end{align}
\end{lem}

The next result \cite[Page 13]{bl}, will be used here.

\begin{lem}\label{ode}
The local solution is non global if we assume that, for some positive real numbers $t_0,\alpha>0$, yields

\begin{align}
M(t)\leq-\alpha\int_{t_0}^t\|\Delta v(s)\|^2\,ds,\quad\forall t\geq t_0.    
\end{align}

\end{lem}

The next collection of radial identities will play an important role in the forthcoming developments.

 \begin{align}
\frac{\partial^2}{\partial x_j\partial x_k}&:=\partial_{jk}=\Big(\frac{\delta_{jk}}r-\frac{x_jx_k}{r^3}\Big)\partial_r+\frac{x_jx_k}{r^2}\partial_r^2;\label{symm}\\
\Delta&=\partial_r^2+\frac{N-1}r\partial_r,\quad \nabla=\frac x{r}\partial_r.\label{sym1}
\end{align}

Indeed, using \eqref{symm}, we have

\begin{align}
  \partial_{jk}\chi\partial_{ik}v\partial_{ij}\bar v
  &=\Big[\Big(\frac{\delta_{jk}}r-\frac{x_jx_k}{r^3}\Big)\partial_r\chi+\frac{x_jx_k}{r^2}\partial_r^2\chi\Big]\partial_{ik}v\partial_{ij}\bar v\nonumber\\
    &=\frac{\partial_r\chi}{r}\sum_k|\nabla v_k|^2+\Big(\frac{\partial_r^2\chi}{r^2}-\frac{\partial_r\chi}{r^3}\Big)\sum_k|x\cdot\nabla v_k|^2\label{5.1}.
\end{align}

We plug \eqref{5.1} in \eqref{mrw}, but also we use \eqref{sym1}, to get

\begin{align}
\frac{\partial}{\partial t}M
&=-2\int_{\R^N}\Big(2\partial_{jk}\Delta\chi\partial_jv\partial_k\bar v-\frac12(\Delta^3\chi)|v|^2+\Delta^2\chi|\nabla v|^2\nonumber\\
&-4\Big(\frac{\partial_r\chi}{r}\sum_k|\nabla v_k|^2+\Big(\frac{\partial_r^2\chi}{r^2}-\frac{\partial_r\chi}{r^3}\Big)\sum_k|x\cdot\nabla v_k|^2\Big)\nonumber\\
&+\frac{q-1}{1+q}\Big[\partial_r^2\chi+\big({N-1}-\frac{2b}{q-1}\big)\frac{\partial_r\chi}r\Big]|v|^{1+q}|x|^{b}\Big)\,dx.\label{5.2}
\end{align}

Now, we let the particular choice

\begin{align}
\chi:r\to\left\{
\begin{array}{ll}
r^2,\quad\mbox{if}\quad 0\leq r\leq1;\\
0,\quad\mbox{if}\quad  r\geq10.
\end{array}
\right.    
\end{align}

So, one has on the unit ball of $\R^N$,

\begin{align}
\chi_{ij} =2\delta_{ij},\quad\Delta\chi =2N\quad\mbox{and}\quad \partial^\gamma\chi= 0 \quad\mbox{for}\quad |\gamma| \geq 3.\label{pr}    
\end{align}

By \cite[Lemma 2.1]{vdd2}, we have

\begin{align}\label{prpr}
\max\{\frac{\chi_R'}r-2, \chi_R''-\frac{\chi_R'}r\}\leq0,\quad\partial^\gamma\chi_R\lesssim R^{2-|\gamma|}, \quad\mbox{for any}\quad |\gamma|\leq6.
\end{align}

Let, for $R>0$, the notation

\begin{align}
    \chi_R:=R^2\chi\big(\frac{|\cdot|}R\big)\quad\mbox{and}\quad M_R:=M_{\chi_R}.
\end{align}

Thus, by \eqref{5.2}, it follows that

\begin{align}
\frac{\partial}{\partial t}M_R
&=8\|\Delta v\|^2-2\int_{\R^N}\Big(2\partial_{jk}\Delta\chi_R\partial_jv\partial_k\bar v-\frac12(\Delta^3\chi_R)|v|^2+\Delta^2\chi_R|\nabla v|^2\nonumber\\
&-4\Big(\frac{\partial_r\chi_R}{r}-2\Big)\sum_k|\nabla v_k|^2-4\Big(\frac{\partial_r^2\chi_R}{r^2}-\frac{\partial_r\chi_R}{r^3}\Big)\sum_k|x\cdot\nabla v_k|^2\nonumber\\
&+\frac{q-1}{1+q}\Big[\big(\partial_r^2\chi_R-2)+\big({N-1}-\frac{2b}{q-1}\big)\big(\frac{\partial_r\chi_R}r-2\big)\Big]|v|^{1+q}|x|^{b}\nonumber\\
&+\frac{q-1}{1+q}\Big[2+2\big({N-1}-\frac{2b}{q-1}\big)\Big]|v|^{1+q}|x|^{b}\Big)\,dx.\label{5.3}
\end{align}

We rewrite \eqref{5.3} as follows

\begin{align}
\frac{\partial}{\partial t}M_R
&=8\Big(\|\Delta v\|^2-\frac{D}{1+q}\int_{\R^N}|v|^{1+q}|x|^b\,dx\Big)\nonumber\\
&-2\int_{\R^N}\Big(2\partial_{jk}\Delta\chi_R\partial_jv\partial_k\bar v-\frac12(\Delta^3\chi_R)|v|^2+\Delta^2\chi_R|\nabla v|^2\nonumber\\
&-4\Big(\frac{\partial_r\chi_R}{r}-2\Big)\sum_k|\nabla v_k|^2-4\Big(\frac{\partial_r^2\chi_R}{r^2}-\frac{\partial_r\chi_R}{r^3}\Big)\sum_k|x\cdot\nabla v_k|^2\nonumber\\
&+\frac{q-1}{1+q}\Big[\big(\partial_r^2\chi_R-2)+\big({N-1}-\frac{2b}{q-1}\big)\big(\frac{\partial_r\chi_R}r-2\big)\Big]|v|^{1+q}|x|^{b}\Big)\,dx.\label{5.4}
\end{align}

Now, using \eqref{prpr} via the conservation of the mass, we have

\begin{align}
(II)
&:=\Big|\int_{\R^N}\Big(2\partial_{jk}\Delta\chi_R\partial_jv\partial_k\bar v-\frac12(\Delta^3\chi_R)|v|^2+\Delta^2\chi_R|\nabla v|^2\Big)\,dx\Big|\nonumber\\
&\lesssim R^{-2}\|\nabla v\|^2+R^{-4}\|v\|^2\nonumber\\
&\lesssim R^{-2}\|\nabla v\|^2+R^{-4}.\label{5.5}
\end{align}

Moreover, since $supp\Big[\big(\partial_r^2\chi_R-2)+\big({N-1}-\frac{2b}{q-1}\big)\big(\frac{\partial_r\chi_R}r-2\big)\Big]\subset \{|x|\geq R\}$, by \eqref{frcs}, we have

\begin{align}
(IV)
&:=\Big|\int_{\R^N}\Big[\big(\partial_r^2\chi_R-2)+\big({N-1}-\frac{2b}{q-1}\big)\big(\frac{\partial_r\chi_R}r-2\big)\Big]|v|^{1+q}|x|^{b}\,dx\Big|\nonumber\\
&\lesssim \int_{|x|\geq R}|x|^{b-\frac{(N-1)(q-1)}2}\big(|x|^\frac{N-1}2|v|\big)^{q-1}|v|^2\,dx\nonumber\\
&\lesssim R^{-(\frac{(N-1)(q-1)}2-b)}\|\nabla v\|^{\frac{q-1}2}.\label{5.6}
\end{align}

We gather \eqref{5.4}, \eqref{5.5} and \eqref{5.6}, via \eqref{prpr} and Young estimate with 

\begin{align}
    \|\nabla\cdot\|^2\lesssim\|\cdot\|\,\|\Delta\cdot\|\label{ntr},    
    \end{align}

to get for $q<9$,

\begin{align}
\frac{\partial}{\partial t}M_R
&\lesssim\|\Delta v\|^2-\frac{D}{1+q}\int_{\R^N}|v|^{1+q}|x|^b\,dx+R^{-1}\big(R^{-1}\|\nabla v\|^2\big)+R^{-4}+R^{-(\frac{(N-1)(q-1)}2-b)}\|\nabla v\|^{\frac{q-1}2}\nonumber\\
&\lesssim D\texttt{E}(v_0)-(D-2)\|\Delta v\|^2+\big(R^{-2}+R^{-2(N-1-\frac{2b}{q-1})}\big)\|\nabla v\|^4+R^{-2}+R^{-\frac{4}{9-q}(\frac{(N-1)(q-1)}2-b)}\nonumber\\
&\lesssim D\texttt{E}(v_0)-(D-2)\|\Delta v\|^2+\big(R^{-2}+R^{-2(N-1-\frac{2b}{q-1})}\big)\|\Delta v\|^2+R^{-2}+R^{-\frac{4}{9-q}(\frac{(N-1)(q-1)}2-b)}.\label{5.7}
\end{align}

Moreover, for $q=9$, yields

\begin{align}
\frac{\partial}{\partial t}M_R
&\lesssim D\texttt{E}(v_0)-(D-2)\|\Delta v\|^2+\big(R^{-2}+R^{-(4(N-1)-b)}\big)\|\Delta v\|^2+R^{-2}.\label{5.8}
\end{align}

\begin{itemize}
    \item Mass-super-critical regime. Thanks to \eqref{ss}, there is $0<\rho<1$ such that

\begin{align}
    {\texttt{E}(v_0)}^{s_c}{\texttt{M}(v_0)}^{2-s_c}<[(1-{\rho})\texttt{E}(\zeta)]^{s_c}\texttt{M}(\zeta)^{2-s_c}.\label{5.9}
\end{align}

Now, by \eqref{5.9} with \eqref{ss2} via Pohozaev identities, we write

\begin{align}
\|\Delta v(t)\|^2
&>\|\Delta\zeta\|^2\Big(\frac{\texttt{M}(\zeta)}{\texttt{M}(v_0)}\Big)^\frac{2-s_c}{s_c}\nonumber\\
&=\frac D{D-2}\texttt{E}(\zeta)\Big(\frac{\texttt{M}(\zeta)}{\texttt{M}(v_0)}\Big)^\frac{2-s_c}{s_c}\nonumber\\
>&\frac D{(1-{\rho})(D-2)}\texttt{E}(v_0).\label{5.10}
\end{align}

Hence, by \eqref{5.10}, yields $ D\texttt{E}(v_0)-(D-2)\|\Delta v\|^2\lesssim -\|\Delta v\|^2$ and thus, using \eqref{5.7} or \eqref{5.8} via \eqref{5.10}, yields for $R\gg1$,

\begin{align}
\frac{\partial}{\partial t}M_R
&\lesssim -\|\Delta v\|^2.\label{5.11'}
\end{align}

Finally, the proof is ended thanks to Lemma \ref{ode}, via \eqref{5.11'}.

\item Mass-critical regime. By \eqref{5.7} and the equality $D=2$, yields

\begin{align}
\frac{\partial}{\partial t}M_R
&\lesssim \texttt{E}(v_0)+\big(R^{-2}+R^{-2(N-1-\frac{2b}{q-1})}\big)\|\Delta v\|^2+R^{-2}+R^{-\frac{4}{9-q}(\frac{(N-1)(q-1)}2-b)}.\label{5.11}
\end{align}

If $T^{\max}<\infty$, we are done. Otherwise, assume by contradiction that 

\begin{align}
    T^{\max}=\infty\quad\mbox{and}\quad \sup_{t\geq0}\|\Delta v(t)\|<\infty.\label{5.12}
\end{align}

Hence, letting $R\gg1$ in \eqref{5.11} via the fact that energy is negative and by use of \eqref{5.12}, we get

\begin{align}
    \frac{\partial}{\partial t}M_R
&\lesssim -1.\label{5.13}
\end{align}

Thus, integrating \eqref{5.13}, we see that there is $t_0>0$, such that 

\begin{align}
M_R<0,\quad\mbox{on}\quad (t_0,\infty).\label{5.15}
\end{align}

Moreover, by H\"older estimate via \eqref{mrwtz}, we write

\begin{align}
 -\|\nabla\chi_R\|_\infty\|v_0\|^\frac32\sqrt{\|\Delta v(t)\|}   \leq M_R(t)\leq -c(t-t_0),\quad\forall\, t>t_0.\label{5.14}
\end{align}

Since \eqref{5.14} implies that 

\begin{align}
 {\|\Delta v(t)\|}   \gtrsim (t-t_0)^2,\quad\forall\, t>t_0,\label{5.16}
\end{align}

which contradicts \eqref{5.15}, the proof is achieved.
\end{itemize}

\section{Appendix: compact Sobolev embedding}

In this section, we prove the inhomogeneous compact Sobolev injection, given in Lemma \ref{compact}. Take $(v_n)\in\big( H^2_{rd}\big)^\mathbb{N}$ such that 
    
    \begin{align}
        \sup_n\|v_n\|_{H^2}<\infty.\label{bnd}
    \end{align}
    
     Hence, it is sufficient to establish that $\displaystyle\lim_n\int_{\R^N}|v_n(x)|^{1+q}|x|^{b}\,dx= 0$, under the assumption $v_n \rightharpoonup0$  weakly to zero in $H^2.$ Letting $\delta>0$, it follows that by \eqref{gag2} and \eqref{bnd},

     \begin{align}
\int_{|x|\leq\delta}|v_n(x)|^{1+q}\,|x|^{b}\,dx
&\leq \delta^b\int_{\R^N}|v_n(x)|^{1+q}\,dx\nonumber\\
&\leq \delta^b\|v_n\|^{E}\|\Delta v_n\|^D\,\nonumber\\
&\leq C\delta^b.\label{cp1}         
     \end{align}

Additionally, using the Strauss inequality \eqref{frcs} in conjunction with the Rellich theorem, yields

\begin{align}
\int_{{\delta}\leq|x|\leq\frac1{\delta}}|v_n(x)|^{1+q}\,|x|^{b}\,dx
&\leq C\|v_n\|_{L^\infty({\delta}\leq|x|\leq\frac1{\delta})}^{q-1}\int_{{\delta}\leq|x|\leq\frac1{\delta}}|v_n(x)|^2\,dx\nonumber\\
&\leq C\delta^{-\frac{N-1}{2}}\|v_n\|_{H^1}^{q-1}\int_{{\delta}\leq|x|\leq\frac1{\delta}}|v_n(x)|^2\,dx\nonumber\\
&\to0,\quad\mbox{as}\quad n\to\infty.    \label{cp2}
\end{align}

Furthermore, by \eqref{frcs} and the fact that $1+\frac{2b}{N-1}<q$, we get

\begin{align}
\int_{|x|\geq\frac1{\delta}}|v_n(x)|^{1+q}\,|x|^{b}\,dx
&=\int_{|x|\geq\frac1{\delta}}|x|^{b-(q-1)\frac{N-1}2}\big(|x|^{\frac{N-1}2}|v_n(x)|\big)^{q-1}|v_n(x)|^2\,dx\nonumber\\
&\leq C\|v_n\|_{H^1}^{q-1}\int_{|x|\geq\frac1\delta}|v_n(x)|^2\,|x|^{b-(q-1)\frac{N-1}2}\,dx\nonumber\\
&\leq C\delta^{(q-1)\frac{N-1}2-b}.\label{cp3}
\end{align}

We collect \eqref{cp1}, \eqref{cp2} and \eqref{cp3} to close the proof by taking $0<\delta\ll1.$

\section*{Data availability statement}

On behalf of all authors, the corresponding author states that there is no conflict of interest. No data-sets were generated or analyzed during the current study.



\end{document}